\documentclass[onecolumn]{article}
\usepackage{amsmath,amssymb,amsthm}
\usepackage{graphicx} 
\usepackage{fullpage}

\theoremstyle{definition}
\newtheorem{theorem}{Theorem}[section] 
\newtheorem{remark}{Remark}[section] 

\newtheorem{lemma}{Lemma}[section]

\numberwithin{equation}{section}

\newcommand{\length}[1]{\left\lvert #1\right\rvert}
\newcommand{\abs}[1]{\bigl\lvert #1\bigr\rvert}
\newcommand{\inv}[1]{#1^{-1}}
\newcommand{\set}[1]{\{#1\}}

\newcommand{\nn}{\mathbb{N}}
\newcommand{\goesto}{\rightarrow} 
\newcommand{\struct}[1]{\mathcal{#1}} 

\DeclareMathOperator{\Exp}{E} 
\DeclareMathOperator{\Var}{Var} 

\title{
  Asymptotic normality of integer compositions inside a rectangle
} 
\author{Steffen Eger\\ \emph{Carnegie Mellon University}, \emph{School of Computer
  Science}\\ {\small \texttt{seger@cs.cmu.edu}}}
\date{}

\begin{document}
\maketitle

\begin{abstract}
Among all restricted integer compositions with at most $m$ parts,
each of which has size at most $l$, choose one uniformly at
random. Which integer does this composition represent? In the current
note, we show that underlying distribution is, for large $m$ and $l$,
approximately normal with mean value $\frac{ml}{2}$. 
\end{abstract}

\section{Introduction}
An integer composition of a nonnegative integer $n$ is, informally, a
way of writing 
$n$ as a sum of nonnegative integers $\pi_1,\dotsc,\pi_k$, for
some $k\ge 0$. Let $h_{l,m}(n)$ denote the number of integer
compositions of the nonnegative integer $n$ with at most $m$ parts, 
each of which has size at most $l$ (`compositions inside a rectangle').
Recently, Sagan (2009) \cite{sagan} has shown that
the sequence 
\begin{align*}
  h_{l,m}:= \bigl(\,h_{l,m}(0),\dotsc,h_{l,m}(lm)\,\bigr),
\end{align*}
is unimodal. In Figure \ref{fig:hlm}, we plot this sequence for $l=2$,
$m=5$; $l=6$, $m=5$; and $l=6$, $m=20$. 
\begin{figure*}[!ht]
    \centering
    \includegraphics[scale=0.12]{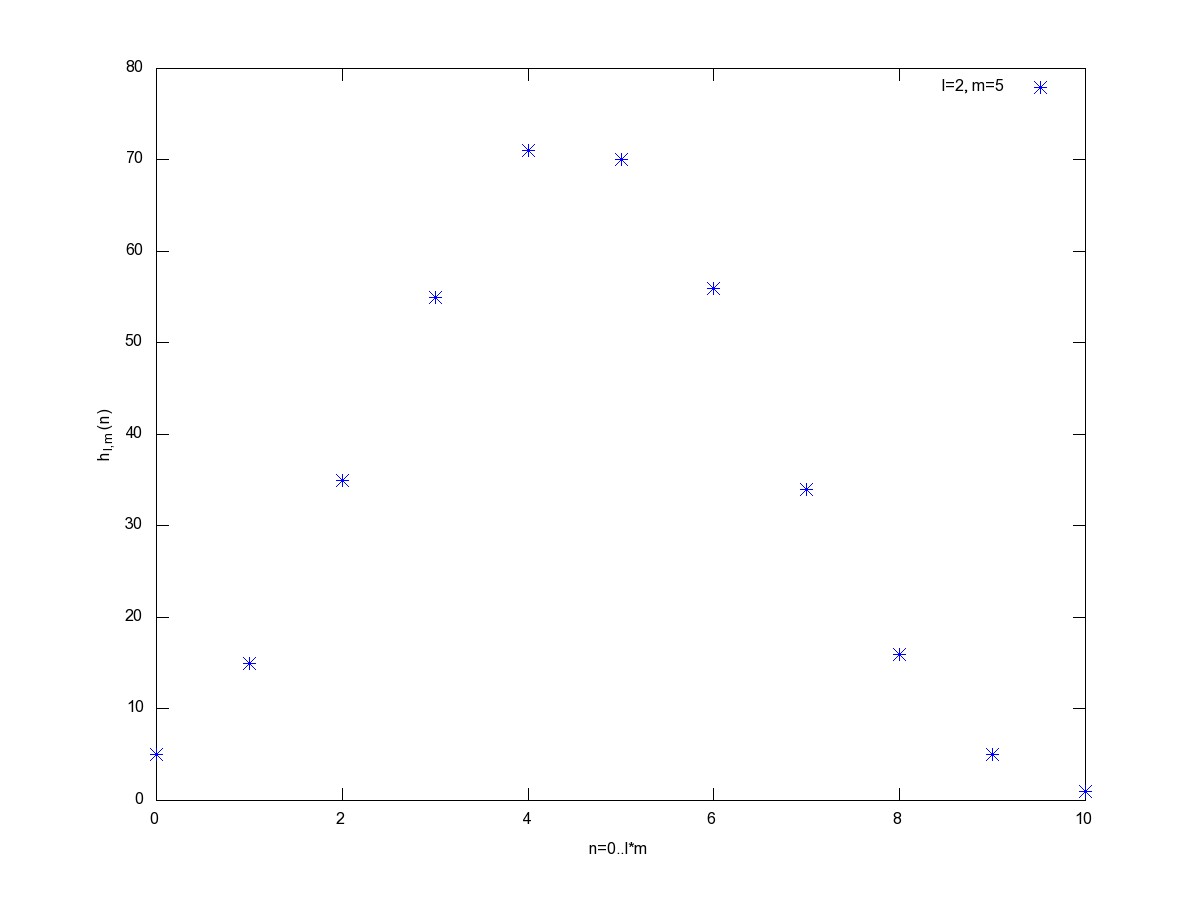}
    \includegraphics[scale=0.12]{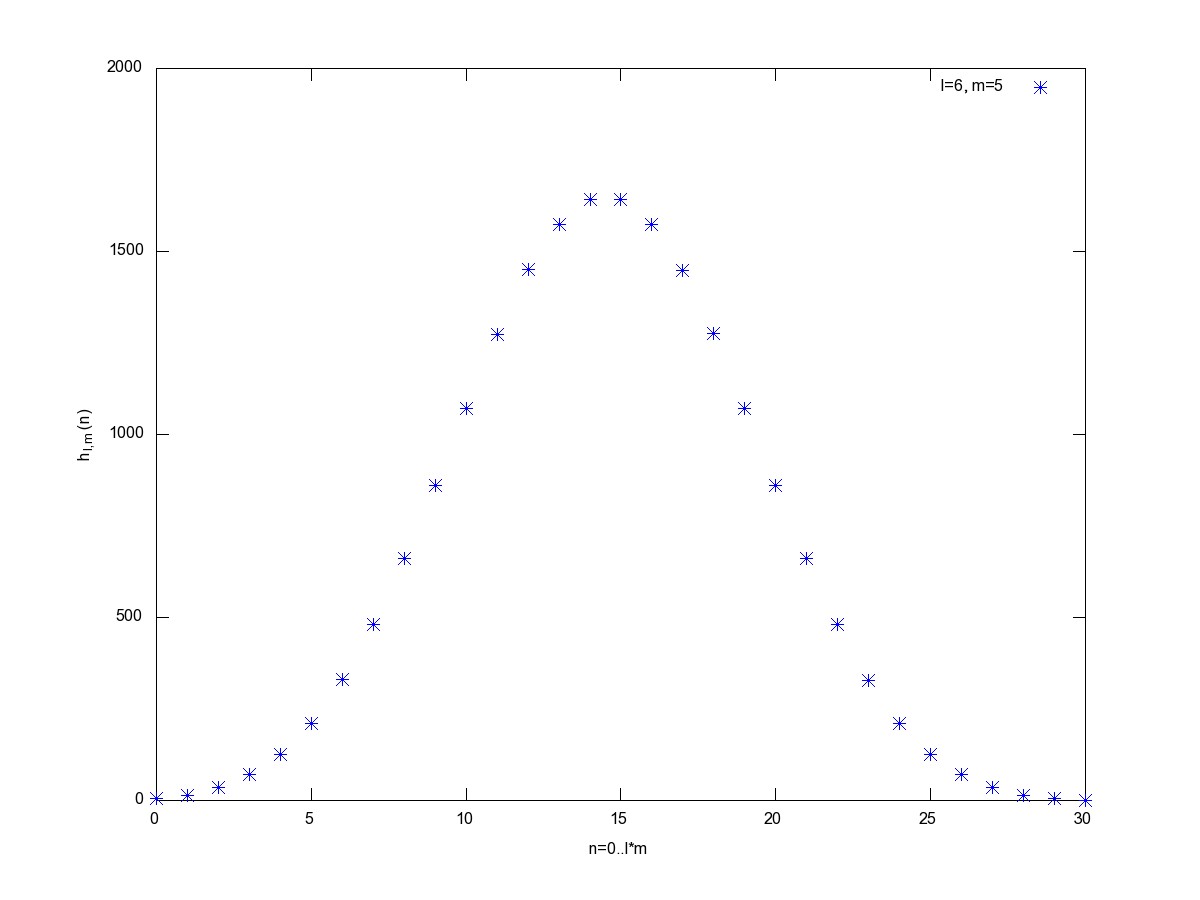}
    \includegraphics[scale=0.12]{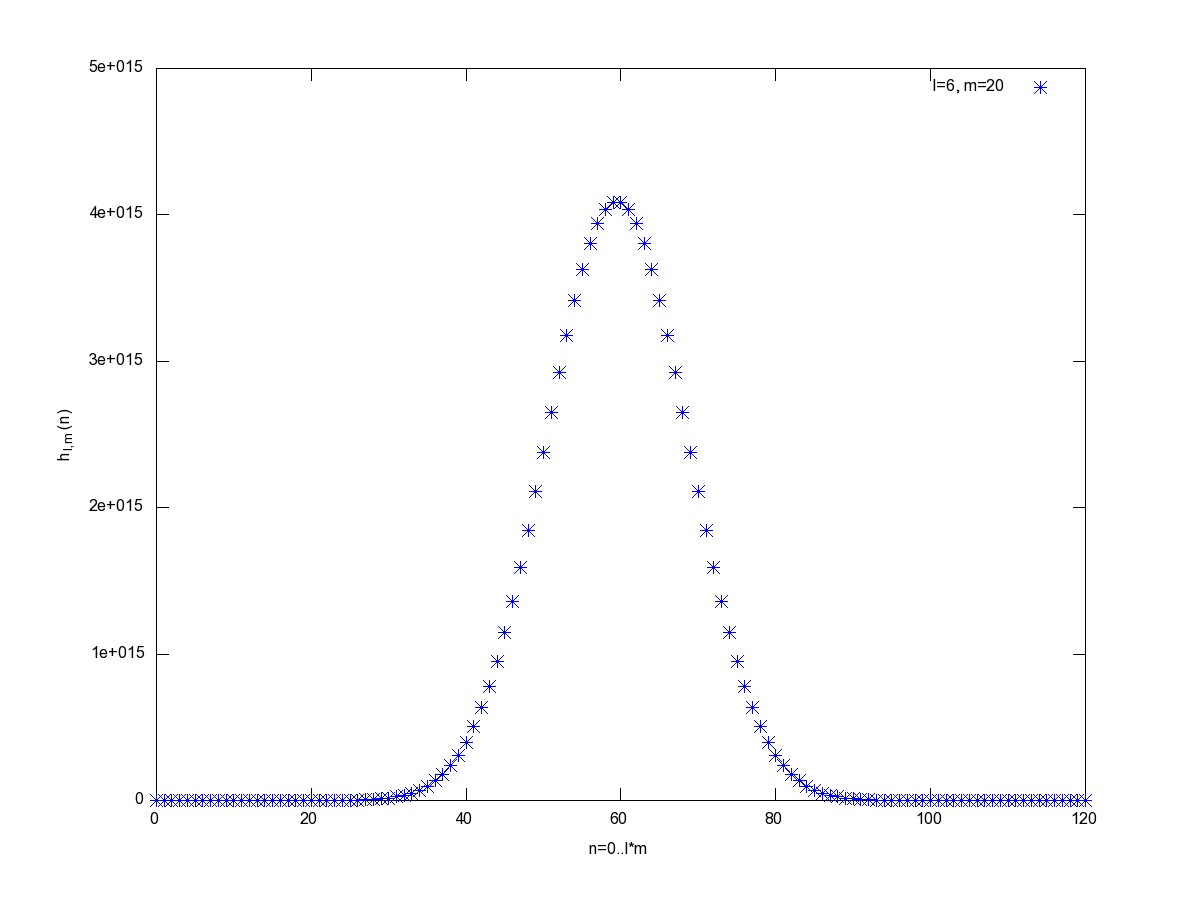}
    \caption
{ 
The sequences $h_{l,m}(0),\dotsc,h_{l,m}(lm)$ for $l=2$, $m=5$ (left),
$l=6$, $m=5$ (middle) and $l=6$, $m=20$ (right). 
}
\label{fig:hlm}
  \end{figure*}
Apparently, as $l$ and $m$ increase, $h_{l,m}$ looks more and more
`Gaussian'. This suggests a probabilistic interpretation of
$h_{l,m}(n)$, 
according to which
the normalized values
$\frac{h_{l,m}(n)}{\sum_{i=0}^{lm}h_{l,m}(i)}$, $n=0,\dotsc,lm$, denote the
probabilities that a uniform randomly chosen integer composition with
at most $m$ parts, each of which has size at most $l$, represents the
integer $n$. 
In the current note, we show that these probabilities follow,
for large $l$ and $m$, 
approximately
a normal distribution with mean
value $\frac{lm}{2}$ and variance $m\frac{(l+1)^2-1}{12}$. 

Thereby, we first define \emph{multinomial triangles} as a
generalization of Pascal's triangle and characterize their entries,
\emph{polynomial coefficients}, as generalizations of the well-studied
binomial coefficients (Section \ref{sec:triangles}), whereupon we outline
a recently found relationship between polynomial coefficients and
specificially restricted integer compositions (Section
\ref{sec:integer}). 
The latter, with various types of restrictions,
have attracted much attention in recent years (cf. \cite{chinn},
\cite{eger}, \cite{heubach},
\cite{hitczenko}, \cite{malandro}, \cite{schmutz},
\cite{shapcott}). For example, Malandro \cite{malandro} determines
asymptotic formulas 
for $L$-restricted integer compositions --- $L$ being an arbitrary
finite set --- and Shapcott \cite{shapcott} and Schmutz and Shapcott
\cite{schmutz} find a lognormal distribution for part products of
restricted integer compositions. 
Hitczenko and Stengle \cite{hitcz2} derive the expected number of
distinct part sizes of unrestricted random compositions. 
Restricted and unrestricted integer
compositions have a variety of applications, ranging from the theory of
patterns \cite{heubach2} to monotone paths in two-dimensional lattices
(\cite{kimberling}), alignments between strings (\cite{eger4}), and
the distribution of the sum of discrete integer-valued random
variables (\cite{eger2}). 

Then, in Section \ref{sec:main}, we state our main theorem, asymptotic
normality of compositions inside a rectangle, which we prove in
Section \ref{sec:proof}. In the conclusion, we discuss
generalizations of the analyzed setting where part sizes are
restricted to lie within 
arbitrary finite sets. 

While our main result, perceived rightly, might be considered not very
surprising, the steps that lead to it (Lemmas \ref{lemma:exact} to
\ref{lemma:approx}) may be judged interesting on their own (and are
certainly novel) because they specify the exact distribution of the
random variable $X_{l,m}$ that sums the parts of a randomly chosen
integer composition from a rectangle of size $l\times m$, and give an
elegant characterization of it in terms of the distribution of the sum
of independent uniform random variables and an ``error term'' that
quadratically tends toward zero.

\section{Multinomial triangles and polynomial coefficients}\label{sec:triangles}
In generalization to binomial triangles, $(l+1)$-nomial
triangles, $l\ge 0$, are defined in the following way. Starting with a
$1$ in row zero, construct an entry in row $k$, $k\ge 1$, by adding the
overlying $(l+1)$ entries in row $(k-1)$ (some of these entries are
taken as zero if not defined); thereby, row $k$ has
$(kl+1)$ entries. For example, the monomial ($l=0$), binomial ($l=1$),
trinomial ($l=2$) and
    quadrinomial triangles ($l=3$) start as follows, 
    \begin{table}[!h]
    \begin{tabular}{r}
      1\\ 1\\ 1\\ 1
    \end{tabular}\hspace{0.5cm}
    \begin{tabular}{rrrr}
      1\\
      1 & 1\\
      1 & 2 & 1\\
      1 & 3 & 3 & 1\\
    \end{tabular}\hspace{0.5cm}
    \begin{tabular}{rrrrrrr}
      1\\
      1 & 1 & 1\\
      1 & 2 & 3 & 2 & 1\\
      1 & 3 & 6 & 7 & 6 & 3 & 1\\
    \end{tabular}\hspace{0.5cm}
    \begin{tabular}{rrrrrrrrrr}
      1\\
      1 & 1 & 1 & 1\\
      1 & 2 & 3 & 4 & 3 & 2 & 1\\
      1 & 3 & 6 & 10 & 12 & 12 & 10 & 6 & 3 & 1\\
    \end{tabular}
    \end{table}

In the $(l+1)$-nomial triangle, entry $n$, $0\le n\le kl$, in row $k$, which we denote
by $\binom{k}{n}_{l+1}$ and refer to as \emph{polynomial coefficient}
(cf. Caiado (2007) \cite{caiado}, Comtet (1974) \cite{comtet}),
has the following interpretation. It 
is the coefficient of $x^n$ in the expansion of
\begin{align}\label{eq:multinomial_coeff1}
  (1+x+x^2+\dotsc+x^l)^k = \sum_{n=0}^{kl} \binom{k}{n}_{l+1}x^n.
\end{align}
Also note that, by its definition, $\binom{k}{n}_{l+1}$ satisfies the
following recursion
\begin{align}
  \binom{k}{n}_{l+1} = \sum_{j=0}^l \binom{k-1}{n-j}_{l+1}.
\end{align}

\section{Integer compositions and polynomial coefficients}\label{sec:integer}
  An {\bf integer composition} of a nonnegative integer $n$ is a tuple
  $\pi=(\pi_1,\dotsc,\pi_k)$, $k\ge 0$, of nonnegative integers such that 
  \begin{align*}
    n = \pi_1+\dotsc+\pi_k
  \end{align*}
  where the $\pi_i$'s are called \emph{parts}, and $k$ is the
  \emph{number of parts}.\footnote{Compositions where some parts are
    allowed to be zero are sometimes called \emph{weak compositions}.}
  Let $\struct{C}(n,k,a,b)$
  denote the set of 
  restricted compositions of $n$ into $k$ parts $\pi_i$ with $a\le
  \pi_i\le b$, where 
  $a,b\in\nn\cup\set{\infty}$ such that $0\le a\le b$, and let
  $c(n,k,a,b)$ denote its size,
  $c(n,k,a,b)=\length{\struct{C}(n,k,a,b)}$.  
  For example, for $n=5$, $k=2$, $a=0$, $b=\infty$, we have
  \begin{align*}
    5 = 5+0=0+5=4+1=1+4=3+2=2+3,
  \end{align*}
  and thus $c(5,2,0,\infty)=6$. 

  The
  following results are well-known. 
  \begin{align}
    c(n,k,0,\infty) &= \binom{n+k-1}{k-1}\label{eq:1}\\
    c(n,k,1,\infty) &= \binom{n-1}{k-1}\label{eq:2}\\
    c(n,k,a,\infty) &= c(n-ka,k,0,\infty) = \binom{n-ka+k-1}{k-1}\label{eq:3}.
  \end{align}

  Moreover, in recent work, Eger (2012) \cite{eger} has shown,
  more generally, a simple relationship between the number of
  restricted integer 
  compositions and polynomial coefficients, namely, 
  \begin{align}\label{eq:eger}
    c(n,k,a,b) &= \binom{k}{n-ka}_{b-a+1}.
  \end{align}

\section{Main theorem}\label{sec:main}
Let $m$ be a positive integer and let $l$ be a nonnegative integer. 
Denote by $h_{l,m}(n)$ the number of integer compositions of the
integer $n$ with at most $m$ parts $p$, each of which has size at most
$l$, i.e. $0\le p\le l$. 
Let $X_{l,m}$ be the
random variable that takes on the integer $n$, for $0\le n\le lm$,
with probability
\begin{align*}
  \frac{h_{l,m}(n)}{\sum_{i=0}^{lm} h_{l,m}(i)}.
\end{align*}
\begin{theorem}\label{theorem:main}
  Let $\mu_{l,m}=\frac{ml}{2}$ and let $\sigma_{l,m}^2 =
  \frac{(l+1)^2-1}{12}$. Then
  \begin{align*}
    \frac{X_{l,m}-m\mu_{l,m}}{\sigma_{l,m} \sqrt{m}}\goesto
      \struct{N}(0,1) \quad\text{as } l,m\goesto\infty.
\end{align*}
\end{theorem}
Our strategy for proving Theorem \ref{theorem:main} is as
follows. First, we determine the exact distribution of $X_{l,m}$ in
Lemma \ref{lemma:exact}. Then we derive the exact distribution of the
sum of $m$ independently and uniformly distributed random variables in
Lemma \ref{lemma:iid}, which is, by the Central Limit Theorem,
asymptotically a normal distribution. Next, Lemmas \ref{lemma:central}
and \ref{lemma:asymptotic} provide inequalities and upper bounds that
we require in Lemma \ref{lemma:approx}, where we show that the distribution
of $X_{l,m}$ can be represented, roughly, as the sum of two parts: the
distribution of the sum $S_1+\dotsc+S_m$ of $m$ independently
distributed uniform 
random variables (derived in Lemma \ref{lemma:iid}) and an ``error term''
that converges quadratically toward zero in $l$. 

\section{Proof of the main theorem}\label{sec:proof}
\begin{lemma}\label{lemma:exact}
  Let $i$, $1\le i\le m$, be the smallest index such that $n\le
  il$. Then,
  \begin{align*}
    P[X_{l,m}=n] =
    \frac{1}{(l+1)^m-1}\frac{l}{l+1}\sum_{j=i}^m\binom{j}{n}_{l+1}.
  \end{align*}
\end{lemma}
\begin{proof}
  By definition, $h_{l,m}(n)=\sum_{j=1}^m
  c(n,j,0,l)=\sum_{j=1}^m\binom{j}{n}_{l+1}$, where the last equality
  follows from \eqref{eq:eger}. Moreover, $c(n,j,0,l)$ is obviously
  zero when $j<i$ since $n>(i-1)l$. Finally, the number of integers
  representable by $j$ parts, each between $0$ and $l$, is obviously
  $(l+1)^j$. Therefore,
  \begin{align*}
    \sum_{i=0}^{lm}h_{l,m}(i) =\sum_{i=0}^{lm} \sum_{j=1}^m
  c(i,j,0,l) = \sum_{j=1}^m\sum_{i=0}^{lm}c(i,j,0,l) =
  \sum_{j=1}^m(l+1)^j = \frac{l+1}{l}{\bigl((l+1)^m-1\bigr)}.
  \end{align*}
  Hence,
  \begin{align*}
    P[X_{l,m}=n]=\frac{h_{l,m}(n)}{\sum_{i=0}^{lm} h_{l,m}(i)}=
    \frac{1}{(l+1)^m-1}\frac{l}{l+1}\sum_{j=i}^m\binom{j}{n}_{l+1}. 
  \end{align*}
\end{proof}

%
%
\begin{lemma}\label{lemma:iid}
  Denote by $S^{(m)}_l$ the sum $S_1+\dotsc+S_m$ of independent
  uniform random variables $S_j$, $j=1,\dotsc,m$, 
  each taking values from the set $\set{0,\dotsc,l}$. The distribution
  of $S^{(m)}_l$ is given by
  \begin{align*}
    P[S^{(m)}_l=n] = \Bigl(\frac{1}{l+1}\Bigr)^m\binom{m}{n}_{l+1}.
  \end{align*}
\end{lemma}
\begin{proof}
  See Caiado \cite{caiado}, Eger \cite{eger2}. 
\end{proof}

\begin{remark}
  Note that the expected value and the variance of $S_l^{(m)}$ in
  Lemma \ref{lemma:iid} are given by
  \begin{align*}
    \Exp[S^{(m)}_l] = m\Exp[S_j] =
    \frac{ml}{2},\quad\quad\Var[S^{(m)}_l] = m\Var[S_j] = 
    m\frac{(l+1)^2-1}{12}.
  \end{align*}
  Also note that, by the Central Limit
  Theorem, the distribution of $S_l^{(m)}$ is asymptotically normal. 
\end{remark}

Now, we prove a fact well-known for binomial coefficients, namely, that
the `central' coefficient majorizes the remaining coefficients
in a given row in the (multinomial) triangle. 
\begin{lemma}\label{lemma:central}
  Let $k\ge 0$ and $l\ge 0$ be integers. For all integers $n$ such
  that $0\le n\le kl$,
  \begin{align*}
    \binom{k}{n}_{l+1} \le \binom{k}{\lfloor\frac{kl}{2}\rfloor}_{l+1}. 
  \end{align*}
\end{lemma}
\begin{proof}
  By the representation of $\binom{k}{n}_{l+1}$ as
  $\binom{k}{n}_{l+1}=\sum_{j=0}^{l}\binom{k-1}{n-j}_{l+1}$ we find for
  $n\ge 1$
  \begin{align}\label{eq:repr}
    \binom{k}{n}_{l+1} =
    \binom{k}{n-1}_{l+1}+\Bigl[\binom{k-1}{n}_{l+1}-\binom{k-1}{n-l-1}_{l+1}\Bigr]. 
  \end{align}
  Moreover, it is easy to show that polynomial coefficients are
  symmetric in the following sense,
  \begin{align*}
    \binom{k}{n}_{l+1} = \binom{k}{kl-n}_{l+1}. 
  \end{align*}
  Therefore it suffices to show that the sequence
  $\binom{k}{0}_{l+1},\binom{k}{1}_{l+1},\dotsc,\binom{k}{\lfloor\frac{kl}{2}\rfloor}_{l+1}$ 
  is non-decreasing. But by \eqref{eq:repr} this easily follows
  inductively, using the row number $k$ as induction
  variable. Importantly, note that, in \eqref{eq:repr}, if $n\le
  \lfloor\frac{kl}{2}\rfloor$, then $\binom{k-1}{n}_{l+1}$ is defined
  and greater 
  than zero for all $k\ge 2$ since then $n\le
  \lfloor\frac{kl}{2}\rfloor\le (k-1)l$.  
\end{proof}

In the following lemma, we write $a_k\sim b_k$
as a short-hand for $\lim_{k\goesto\infty} \frac{a_k}{b_k}=1$. Also
note that the following lemma is a generalization of Stirling's
approximation to the central binomial coefficient. 
\begin{lemma}\label{lemma:asymptotic}
  For all fixed $l$,
  \begin{align*}
  \binom{k}{\lfloor\frac{kl}{2}\rfloor}_{l+1}\sim
  \frac{(l+1)^{k}}{\sqrt{2\pi 
      k\frac{(l+1)^2-1}{12}}}.
  \end{align*}
\end{lemma}
\begin{proof}
  See Eger \cite{eger3}.
\end{proof}

\begin{lemma}\label{lemma:approx}
  For all 
  $l$ and $m$ and
  for all $n$ such that $0\le n\le ml$,
  \begin{align*}
    P[X_{l,m}=n] = \gamma_{l,m}P[S^{(m)}_l=n]+e_{l,m}, 
  \end{align*}
  where $e_{l,m}$ is an ``error term'' that satisfies
  \begin{align*}
    0\le e_{l,m} \le O(l^{-2}) 
  \end{align*}
  and $\gamma_{l,m}$ satisfies
  \begin{align*}
    \gamma_{l,m}=\bigl(1+O(\inv{l})\bigr)^{-1}.
  \end{align*}
\end{lemma}
\begin{proof}
  
  Let $i$, $1\le i\le m$, be the smallest index such that $n\le
  il$. 
  Moreover, define $\alpha_{l,m}$ as
  $\alpha_{l,m}=\frac{1}{(l+1)^m-1}\frac{l}{l+1}$ and note that
  $\alpha_{l,m}=\gamma_{l,m}\frac{1}{(l+1)^m}$, where
  $\gamma_{l,m}=\inv{(1+1/l)}$ (ignoring the $(-1)$ in the
  denominator of $\alpha_{l,m}$). Then
  \begin{align*}
    P[X_{l,m}=n] &= \alpha_{l,m}\sum_{j=i}^m
    \binom{j}{n}_{l+1} =
    \alpha_{l,m}\binom{m}{n}_{l+1}+\alpha_{l,m}\sum_{j=i}^{m-1}\binom{j}{n}_{l+1}
    = 
    \gamma_{l,m}P[S^{(m)}_l=n]+e_{l,m}, 
  \end{align*}
  where we define
  $e_{l,m}=\alpha_{l,m}\sum_{j=i}^{m-1}\binom{j}{n}_{l+1}$. Obviously,
  $e_{l,m}\ge 0$. Moreover, by Lemmas \ref{lemma:central} and
  \ref{lemma:asymptotic} 
  \begin{align}\label{eq:start}
    e_{l,m}&\le
    \alpha_{l,m}\sum_{j=i}^{m-1}\binom{j}{\lfloor\frac{jl}{2}\rfloor}_{l+1}
    \le
    \alpha_{l,m}O(1)\sum_{j=i}^{m-1}\frac{(l+1)^j}{\sqrt{2\pi
        j\frac{(l+1)^2-1}{12}}}. 
  \end{align}
  Now,
  \begin{align*}
    \frac{(l+1)^j}{\sqrt{2\pi j\frac{(l+1)^2-1}{12}}} =
    O(1)\cdot\frac{(l+1)^{j}}{\sqrt{j\bigl((l+1)^2-1\bigr)}},
  \end{align*}
  so that
  \begin{align*}
    \sum_{j=i}^{m-1}\frac{(l+1)^j}{\sqrt{2\pi
        j\frac{(l+1)^2-1}{12}}} =
    \sum_{j=i}^{m-1}O(1)\frac{(l+1)^{j}}{\sqrt{j}\sqrt{(l+1)^2-1}} \le
    O(1)\sum_{j=i}^{m-1}(l+1)^{j-1}
    = O(1)\frac{(l+1)^{i-1}\bigl[(l+1)^{m-i}-1\bigr]}{l},
  \end{align*}
  whence, continuing from \eqref{eq:start},
  \begin{equation}\label{eq:shape}
    \begin{split}
    e_{l,m}&\le \alpha_{l,m}O(1)\sum_{j=i}^{m-1}\frac{(l+1)^j}{\sqrt{2\pi
        j\frac{(l+1)^2-1}{12}}} \le  
    O(1)\frac{(l+1)^{i-2}}{(l+1)^m-1}\bigl[(l+1)^{m-i}-1\bigr]
    \\&\le
    O(1)\Bigl((l+1)^{-2}-(l+1)^{i-m-2}\Bigr)
    \le O(1)(l+1)^{-2}.
    \end{split}
  \end{equation}  
\end{proof}

In Table \ref{table:approx}, we show the decrease of $e_{l,m}$ in Lemma
\ref{lemma:approx} as $l$ increases. Obviously, our bound is apparently
quite well, as in fact $e_{l,m}$ seems to 
approximately quadratically decay in $l$. 
In Figure \ref{fig:xlm_slm}, the distributions of
$X_{l,m}$ and $S_{l}^{(m)}$ for different values of $l$ and $m$ are
plotted. The variable $X_{l,m}$ has a particular distributional shape
that can be 
inferred from the proof of Lemma \ref{lemma:approx}. For small
values $n$ the distribution of $X_{l,m}$ tends to be larger than that
of $S_l^{(m)}$ --- $e_{l,m}$ is relatively larger as can be seen from
Equation \eqref{eq:shape} 
--- while this relation is 
reversed for large $n$.

\begin{table}[!h]
  \centering
  \begin{tabular}{l||rr|rr}
    & \multicolumn{2}{|c|}{$m=10$}  & \multicolumn{2}{|c}{$m=20$} \\\hline\hline
    $l=1$ & $0.0471$ & & $0.0240$ &\\
    $l=2$ & $0.0191$ & $2.46$ & $0.0093$ & $2.57$\\
    $l=4$ & $0.0064$ & $2.94$ & $0.0031$ & $2.96$\\
    $l=8$ & $0.0019$ & $3.25$ & $9.5016\times 10^{-4}$& $3.30$\\
    $l=16$ & $5.5909\times 10^{-4}$ & $3.56$ &  $2.6494\times 10^{-4}$& $3.58$\\
    $l=32$ & $1.4871\times 10^{-4}$ & $3.75$ & $7.0291\times 10^{-5}$& $3.76$\\
    $l=64$ & $3.8399\times 10^{-5}$ & $3.82$ & $1.8126\times 10^{-5}$& $3.87$\\
  \end{tabular}\hspace{0.5cm}
  \caption{
    Maximum over absolute differences
    $\abs{P[X_{l,m}=n]-P[S^{(m)}_l=n]}$, $n=0,\dotsc,lm$, for $m=10$
    and $m=20$ and varying $l$. We also specify the factor of decrease in
    these differences
    between successive $l$ values.
  }
  \label{table:approx}
\end{table}
\begin{figure*}[!ht]
    \centering
    \includegraphics[scale=0.12]{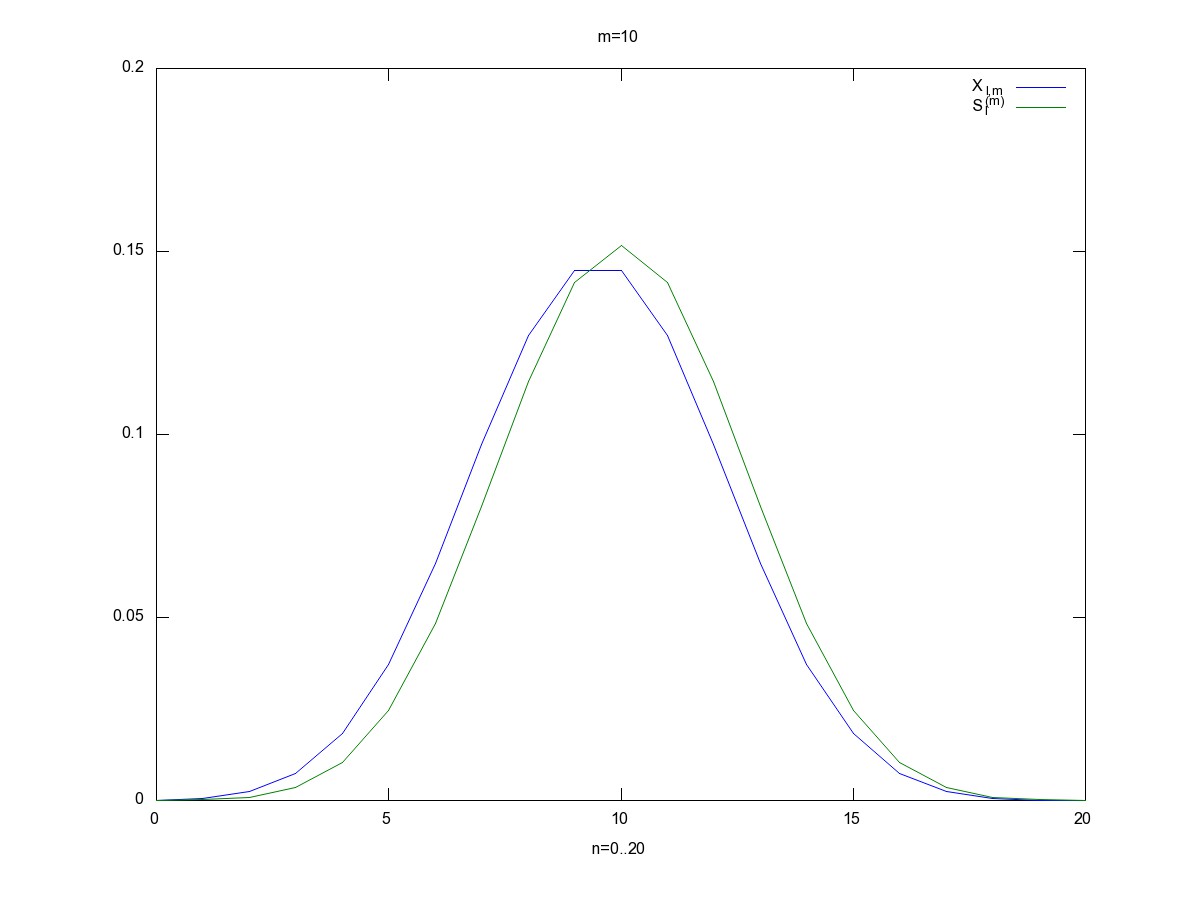}
    \includegraphics[scale=0.12]{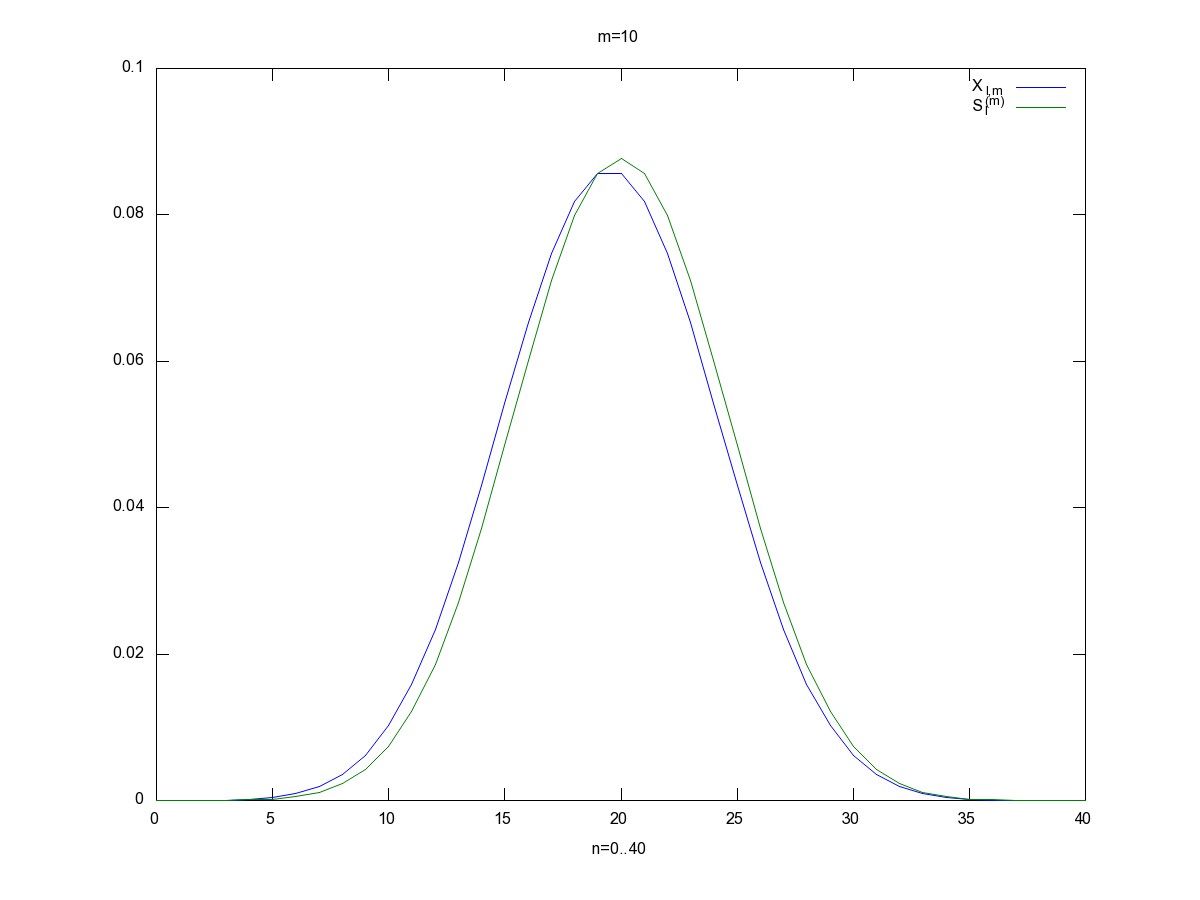}
    \includegraphics[scale=0.12]{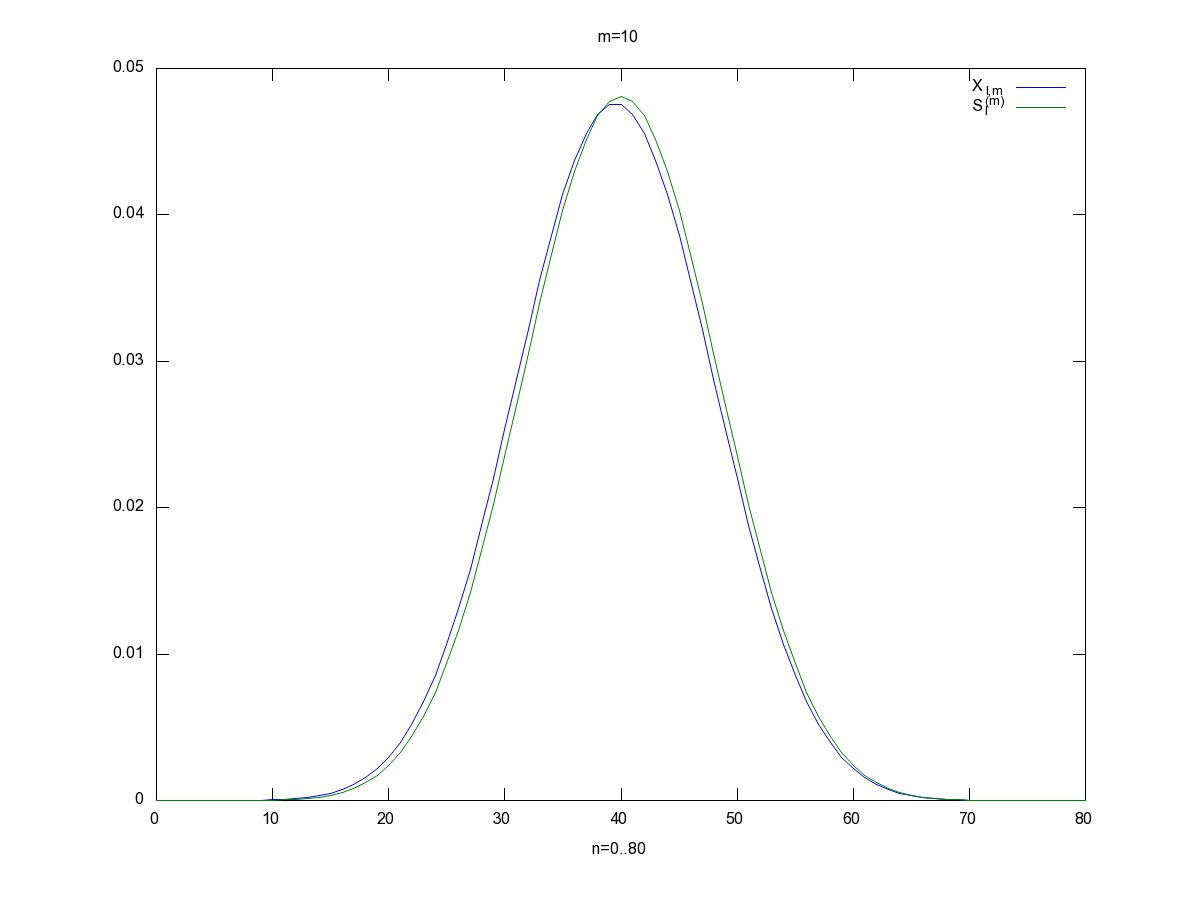}
    \caption
{ 
The distributions of $X_{l,m}$ and $S^{(m)}_l$ for $m=10$ and $l=2$
(left), $l=4$ (middle), and $l=8$ (right). 
}
\label{fig:xlm_slm}
  \end{figure*}

\section{Conclusion}
The choice of the restrictions $0\le p\le l$ for parts $p$ of integer
compositions has, although illustrating a model case, largely been
arbitrary. In fact, similar results as Theorem \ref{theorem:main}
would hold for any finite set 
$L=\set{a_1,\dotsc,a_k}$ as range for part sizes. For
$L=\set{a,a+1,\dotsc,b}$, $0\le a\le b$, we find simple closed form
solutions of the asymptotic distribution of $X_{L,m}$, where we define
$X_{L,m}$ (and other variables such as $S_{L}^{(m)}$) as a
generalization of $X_{l,m}$ above with 
$X_{l,m}=X_{\set{0,\dotsc,l},m}$. For example, in this case,
$S^{(m)}_L$ has exact distribution 
\begin{align*}
  \Bigl(\frac{1}{b-a+1}\Bigr)^m\binom{m}{n-ma}_{b-a+1},
\end{align*}
(cf. Eger (2012) \cite{eger2}) with expected value $\frac{m(a+b)}{2}$
and is, by the Central Limit Theorem, asymptotically normally
distributed. Conversely, the distribution of $X_{L,m}$ allows a
similar representation as in Lemma \ref{lemma:exact}, as a sum of
quantities $\binom{j}{n-ja}_{b-a+1}$ and a normalizing term, from
which we can straightforwardly derive a decomposition of $X_{L,m}$ as
in Lemma \ref{lemma:approx}, with bounds obtained from Lemmas
\ref{lemma:central} and \ref{lemma:asymptotic}.

As a final remark, note that our results entail a `Stirling' like
formula for $h_{l,m}(n)$. By definition
$P[X_{l,m}=n]=\frac{h_{l,m}(n)}{\sum_{i=0}^{lm} h_{l,m}(i)}$, and
equating this quantity at its asymptotic mean value $\frac{ml}{2}$
with the corresponding normal density leads to
\begin{align*}
  h_{l,m}(\frac{ml}{2}) \sim
  \frac{\bigl((l+1)^m-1\bigr)\frac{l+1}{l}}{\sqrt{2\pi m
      \frac{(l+1)^2-1}{12}}}. 
\end{align*} 

\end{document}